\documentclass[12pt]{article}
\usepackage{amsmath}
\usepackage{amssymb, amsfonts}
\usepackage{amsthm}
\usepackage[all,arc]{xy}
\usepackage{enumerate}
\usepackage{fullpage}
\usepackage{graphicx}
\usepackage[lined,algonl,boxed]{algorithm2e}

\newtheorem{lem}{Lemma}
\newtheorem{fact}{Fact}
\newtheorem{thm}{Theorem}

\newtheorem{definition}{Definition}

\title{\textbf{The Entropy Influence Conjecture Revisited}}

\author{
Bireswar Das  \thanks{Indian Institute of Technology Gandhinagar, India. {\tt bireswar@iitgn.ac.in}} 
\and Manjish Pal \thanks{Indian Institute of Technology Gandhinagar, India. {\tt manjish\_pal@iitgn.ac.in}} 
\and Vijay Visavaliya \thanks{Vishwakarma Government Engineering College Gandhinagar, India. {\tt visavaliavijay@gmail.com}} }

\date{}
 
\begin{document}
 
\maketitle

\begin{abstract}
In this paper, we prove that most of the boolean functions, $f : \{-1,1\}^n \rightarrow \{-1,1\}$ 
satisfy the Fourier Entropy Influence (FEI) Conjecture due to Friedgut and Kalai (Proc. AMS'96)\cite{FG96}. The conjecture says that
the Entropy of a boolean function is at most a constant times the Influence of the function. The conjecture
has been proven for families of functions of smaller sizes. O'donnell, Wright and Zhou (ICALP'11)\cite{DWZ11} verified the
conjecture for the family of symmetric functions, whose size is $2^{n+1}$. They are in fact able to prove
the conjecture for the family of $d$-part symmetric functions for constant $d$, the size of whose is $2^{O(n^d)}$. Also it is known that the 
conjecture is true for a large fraction of polynomial sized DNFs (COLT'10)\cite{KLW10}. Using elementary methods we prove that 
a random function with high probability satisfies the conjecture with the constant as $(2 + \delta)$, for any constant $\delta > 0$. 
\end{abstract}

\section{Introduction}
The Entropy Influence Conjecture due to Friedgut and Kalai \cite{FG96} says that for every boolean function $f : \{-1,1\}^n \rightarrow \{-1,1\}$
the following holds,
\[
\displaystyle \sum_{S \subseteq [n]} \widehat{f}(S)^2 \log_2 \frac{1}{\widehat{f}(S)^2} \leq  C \cdot \sum_{S \subseteq [n]} \widehat{f}(S)^2 |S|
\]
for some universal constant $C > 0$ where $\widehat{f}(S)$ is the coefficient of $\chi_S(x)$ in the Fourier expansion of $f$.
The conjecture is of profound importance because of its potential impacts in areas like Learning Theory, Threshold Phenomena in monotone graph properties, metric embeddings  etc. For a detailed description of the impact and background
of the FEI conjecture the reader is recommended to read the Introduction section of \cite{DWZ11} and the blog post by Gil Kalai \cite{GK}.

The conjecture can be verified for simple functions like AND, OR, MAJORITY, Tribes etc. Although
posed about 15 years back we are not aware of a significantly large (of doubly exponential size) family
of functions which satisfies this conjecture. Klivans et al. \cite{KLW10} proved recently that a large fraction of polynomial sized DNF formulae
satisfy the Mansour's conjecture \cite{M94} which in turn implies that FEI conjecture is also true 
for those functions, a class  which has a size of $2^{poly(n)}$. In a recent resurrection of the FEI conjecture,
O'Donnell et al. \cite {DWZ11} proved the conjecture for the family of symmetric functions and $d$-part symmetric
functions for constant $d$. The sizes of these families are $2^{n+1}$ and $2^{O(n^d)}$ respectively, again only 
exponential in size. They also verified it for read-once decision trees which are of exponential size as well.
Thus, one is not aware of an explicit or non-explicit  family of doubly exponential size that satisfies the conjecture.
Very recently, in a note, Keller et al \cite{KMS11} managed to prove a variant of the conjecture for functions 
which have low Fourier weight on characters of large size.

\section{Preliminaries}
In this section we introduce the basic preliminaries of discrete Fourier Analysis which will be 
of interest for us. 

\begin{definition}
Let $f:\{-1,1\}^n \rightarrow \{-1,1\}$ be a boolean function. The \emph{Fourier expansion} of $f$ is written
as 
\[
\displaystyle f(x) = \sum_{S \subseteq [n]} \widehat{f}(S) \chi_S(x)
\]
where $\chi_S(x) = \displaystyle \Pi_{i \in S} x_i$.
\end{definition}

\begin{definition}
Let $f:\{-1,1\}^n \rightarrow \{-1,1\}$ be a boolean function. The \emph{entropy} of $f$ is defined as 
\[
\displaystyle \textbf{H}(f) = \sum_{S \subseteq [n]} \widehat{f}(S)^2 \log_2 \frac{1}{\widehat{f}(S)^2}
\]
\end{definition}

Notice by Parseval's identity, $\sum_{S \subseteq [n]} \widehat{f}(S)^2 = 1$ which implies 
that the Fourier coefficients can be thought of as a probability distribution and
hence $\textbf{H}(f)$ gives us the entropy of that distribution. The following
fact (see page 40 of \cite{CT91}) gives an upper bound on $\textbf{H}(f)$ which will be of use for us later.

\begin{fact} \label{a}
For an arbitrary boolean function $f$, $\textbf{H}(f) \leq n$.
\end{fact}

\begin{definition}
For a boolean function $f:\{-1,1\}^n \rightarrow \{-1,1\}$, $\textbf{Inf}_i(f)$, the \emph{Influence of coordinate $i$} is defined
as 
\[
\displaystyle \textbf{Inf}_i(f) = \sum_{S:i\in S}  \widehat{f}(S)^2
\]
and the \emph{Influence of $f$}, $\textbf{Inf}(f)$, is defined as 
\[
\displaystyle \textbf{Inf}(f) = \sum_{i=1}^n \textbf{Inf}_i =  \sum_{S \subseteq [n]} \widehat{f}(S)^2 |S|
\]
\end{definition}

An equivalent combinatorial interpretation of the influence of the $i$th coordinate is given by $\textbf{Inf}_i=\mbox{Pr}_x[f(x)\neq f(x^{(i)})]$ where $x^{(i)}$ is $x$ with the $i$th coordinate flipped. 

~

We will also use the following well-known facts regarding the Fourier
coefficients.

\begin{fact}
For a boolean function $f:\{-1,1\}^n \rightarrow \{-1,1\}$, the following holds for any subset $S\subseteq [n]$
\begin{itemize}
\item[1.]  $\displaystyle \widehat{f}(S) = \frac{1}{2^{n}} \sum_{x} f(x)\chi_{S}(x)$.
\item[2.]  $\displaystyle \sum_{x} \chi_S(x) = 0$ if $S \neq \phi$, and $2^{n}$ if $S = \phi$.
\end{itemize}
\end{fact}

\section{The Result}

The main contribution of our paper is to prove that there is a 
large (non-explicit) family of functions that satisfies the entropy
influence conjecture. The size of this family is significantly
larger than the size of any of the known families for which
the FEI conjecture is known to be true. More precisely, we show that there
is a family of functions whose size is
\[
 \displaystyle \left[1 - 4\left(1 + \frac{2}{\delta} \right)^2 \frac{1}{2^{n+1}n} \right]\cdot 2^{2^n}
\]
satisfies the conjecture with $C = 2 + \delta$ for any constant $\delta > 0$. Of course, our result
doesn't provide a step towards the resolving the FEI conjecture as it is known that the there
are functions which need the constant $C$ to be at least 4.6 \cite{DWZ11}.

\begin{thm}\label{thm1}
A random function satisfies the FEI  conjecture with high probability with $C = 2 + \delta$,  for any constant $\delta > 0$.

\end{thm}

Consider random function which puts values $1$ or $-1$ independently on every point of $\{1,-1\}^n$ with an equal probability of $1/2$. Clearly, every function is obtained with a probability of $\frac{1}{2^{2^n}}$. Let $H_r$ and $I_r$ be random variables denoting the entropy and influence of a randomly chosen function as above.  We will prove Theorem \ref{thm1} using a simple application of Chebyshev Inequality.
We will use the following two lemmas in our proof.

\begin{lem}\label{lem1}
  $\textbf{E}[I_r]=\frac{n}{2}$.
\end{lem}

\begin{proof}
\begin{eqnarray*}
\textbf{E}[I_r] &=& \textbf{E}\left[ \sum_{S \subseteq [n]} \widehat{f}(S)^2 |S|\right] \\
&=& \sum_{S \subseteq [n]} \textbf{E}[\widehat{f}(S)^2] |S| \\
&=& \sum_{S \subseteq [n]} \frac{1}{2^{2n}} \textbf{E}\left[ \sum_x f(x)^2 \chi_S(x)^2 + \sum_{x \neq y} f(x)f(y)\chi_S(x) \chi_S(y)  \right] |S| 
\end{eqnarray*}  
\begin{eqnarray*}
&=&  \frac{1}{2^{2n}}\sum_{S \subseteq [n]} \left( \sum_x 1 + \sum_{x \neq y} \textbf{E} \left[f(x)f(y)\right] \chi_S(x) \chi_S(y) \right) |S| \\
&=&  \frac{1}{2^{2n}} \sum_{S \subseteq [n]} 2^n |S|+  0  \mbox {    (this follows because for $x \neq y$, $\textbf{E}[f(x)f(y)] = \textbf{E}[f(x)]\textbf{E}[f(y)] = 0$ )} \\
&=&  \frac{1}{2^{n}}  \sum_{k =0 }^n  \binom{n}{k} k = \frac{n 2^{n-1}}{2^n} = \frac{n}{2}
\end{eqnarray*}  
\end{proof}

\begin{lem}\label{lem2}
  $\textbf{Var}[I_r]=\frac{n}{2^{n+1}}$.
\end{lem}
\begin{proof} We have already calculated $\textbf{E}[I_r]$. To calculate, 
\begin{eqnarray*}
\textbf{Var}[I_r] &=& \textbf{E}[I_r^2] - (\textbf{E}[I_r])^2
\end{eqnarray*}
 we need to calculate $\textbf{E}[I_r^2]$.

\begin{eqnarray*}
\textbf{E}[I_r^2] &=& \textbf{E}\left[ \left(\sum_{S \subseteq [n]} \widehat{f}(S)^2 |S|\right)^2 \right] \\
&=& \textbf{E}\left[ \sum_{S_1, S_2 \subseteq [n]}  \widehat{f}(S_1)^2  \widehat{f}(S_2)^2 |S_1| |S_2| \right] \\
&=& \sum_{S_1 , S_2 \subseteq [n]}  \textbf{E}  \left[\widehat{f}(S_1)^2  \widehat{f}(S_2)^2\right]  |S_1| |S_2|
\end{eqnarray*}

\begin{eqnarray*}
 \textbf{E}  \left[\widehat{f}(S_1)^2  \widehat{f}(S_2)^2\right] &=&\frac{1}{2^{4n}}  \textbf{E} \left[\left(\sum_x f(x) \chi_{S_1}(x) \right)^2 \left(\sum_x f(x) \chi_{S_2}(x)\right)^2\right] \\
 &= & \frac{1}{2^{4n}} \textbf{E} \left[ \sum_{x_1,y_1,x_2,y_2} f(x_1)f(y_1)f(x_2)f(y_2) \chi_{S_1}(x_1) \chi_{S_1}(y_1) \chi_{S_2}(x_2) \chi_{S_2}(y_2) \right]\\
 &= & \frac{1}{2^{4n}} \sum_{x_1,y_1,x_2,y_2} \textbf{E} \left[  f(x_1)f(y_1)f(x_2)f(y_2) \right] \chi_{S_1}(x_1) \chi_{S_1}(y_1) \chi_{S_2}(x_2) \chi_{S_2}(y_2) 
\end{eqnarray*}
To calculate the above sum, consider the following sets,
\begin{eqnarray*}
A_1 = \{(x_1,y_1,x_2,y_2) | x_1 = y_1, x_2 = y_2\} \\
A_2 = \{(x_1,y_1,x_2,y_2) | x_1 = x_2, y_1 = y_2\} \\
A_3 = \{(x_1,y_1,x_2,y_2) | x_1 = y_2, x_2 = y_1\} \\
\end{eqnarray*} 
Notice the following properties of $A_1, A_2, A_3$, $|A_1| = |A_2| = |A_3| = 2^{2n}, A_1 \cap A_2 = A_2 \cap A_3 = A_3 \cap A_1 = A_1 \cap A_2 \cap A_3$
and $|A_1 \cap A_2 \cap A_3| = 2^{n}$.
It is easy to verify that if  $(x_1,y_1,x_2,y_2) \notin A_1 \bigcup A_2 \bigcup A_3$, then 
$\textbf{E} \left[ f(x_1)f(y_1)f(x_2)f(y_2) \right]  = 0$.
Otherwise it is $\chi_{S_1}(x_1) \chi_{S_1}(y_1) \chi_{S_2}(x_2) \chi_{S_2}(y_2)$.
Using the above properties and inclusion exclusion principle we have $ \textbf{E}  \left[\widehat{f}(S_1)^2  \widehat{f}(S_2)^2\right]$ equal to 

\begin{eqnarray*}
& & \frac{1}{2^{4n}} \left[ \sum_{A_2}  \chi_{S_1}(x_1) \chi_{S_1}(y_1) \chi_{S_2}(x_2) \chi_{S_2}(y_2) + \sum_{A_3}  \chi_{S_1}(x_1) \chi_{S_1}(y_1) \chi_{S_2}(x_2) \chi_{S_2}(y_2) \right] +  \\
&&   \frac{1}{2^{4n}} \left[ \sum_{A_1}  \chi_{S_1}(x_1) \chi_{S_1}(y_1) \chi_{S_2}(x_2) \chi_{S_2}(y_2)\right ]  - 2 \cdot \frac{1}{2^{4n}} \left[ \sum_{A_1 \cap A_2 \cap A_3}  \chi_{S_1}(x_1) \chi_{S_1}(y_1) \chi_{S_2}(x_2) \chi_{S_2}(y_2)\right ] \\
&=& \frac{1}{2^{4n}} \left[ \sum_{x,y}  \chi_{S_1}(x) \chi_{S_1}(y) \chi_{S_2}(x) \chi_{S_2}(y) + \sum_{x,y}  \chi_{S_1}(x) \chi_{S_1}(y) \chi_{S_2}(y) \chi_{S_2}(x) \right] +  \\
&&   \frac{1}{2^{4n}} \left[\sum_{x,y}  \chi_{S_1}(x) \chi_{S_1}(x) \chi_{S_2}(y) \chi_{S_2}(y) \right ]  - 2 \cdot \frac{1}{2^{4n}} \cdot 2^n \\
&=& \frac{1}{2^{4n}} \left[2 \sum_{x,y}  \chi_{S_1 \Delta S_2}(x) \chi_{S_1 \Delta S_2}(y) + 2^{2n} - 2 \cdot 2^{n}\right].
\end{eqnarray*}
Thus, using the fact that $ \sum_x \chi_{S_1 \Delta S_2}(x) = 0$ if $S_1 \neq S_2$ and $2^n$ otherwise, we have
\[ \textbf{E}  \left[\widehat{f}(S_1)^2  \widehat{f}(S_2)^2\right] = \left\{ 
\begin{array}{l l}
  \frac{1}{2^{4n}} \cdot (2^{2n} - 2\cdot 2^{n}) & \quad \mbox{if $S_1 \neq S_2$.}\\
   \frac{1}{2^{4n}} \cdot (3\cdot 2^{2n} -  2\cdot 2^{n}) & \quad \mbox{otherwise.}\\ \end{array}\right. \]

Therefore,
\begin{eqnarray*}
\textbf{E}[I_r^2] &=& \sum_{S_1, S_2 \subseteq [n]}  \textbf{E}  \left[\widehat{f}(S_1)^2  \widehat{f}(S_2)^2\right]  |S_1| |S_2| \\
&=&  \sum_{S_1= S_2}  \textbf{E}  \left[\widehat{f}(S_1)^2  \widehat{f}(S_2)^2\right]  |S_1| |S_2| +  \sum_{S_1\neq S_2}  \textbf{E}  \left[\widehat{f}(S_1)^2  \widehat{f}(S_2)^2\right]  |S_1| |S_2| \\
&=&\displaystyle   \sum_{S_1= S_2}  \frac{3 \cdot 2^{2n} - 2\cdot 2^{n}}{2^{4n}} |S_1| |S_2| +  \sum_{S_1\neq S_2}  \frac{ 2^{2n} - 2\cdot 2^{n}}{2^{4n}} |S_1| |S_2| \\
&= & \displaystyle  \sum_{S_1= S_2}  \frac{2 \cdot 2^{2n}}{2^{4n}} |S_1| |S_2| + \sum_{S_1, S_2}  \frac{ 2^{2n} - 2\cdot 2^{n}}{2^{4n}} |S_1| |S_2| \\
&=& \displaystyle  \frac{2 \cdot 2^{2n}}{2^{4n}} \sum_{k=0}^{n} \binom{n}{k} k^2 + \frac{ 2^{2n} - 2\cdot 2^{n}}{2^{4n}} \left( \sum_{k=0}^{n} \binom{n}{k} k\right)^2 \\
&=&  \displaystyle \frac{2(n(n-1)2^{n-2} + n2^{n-1})}{2^{2n}} + \left( \frac{1}{2^{2n}} - \frac{2}{2^{3n}} \right)(n2^{n-1})^2 \\
&=& \displaystyle \frac{n}{2^{n+1}} + \frac{n^2}{4}.
\end{eqnarray*}
Hence,
\[
\textbf{Var}[I_r] = \frac{n}{2^{n+1}}.
\]  
\end{proof}
We are now ready to prove Theorem \ref{thm1}.
\begin{proof}(Theorem \ref{thm1})
We will prove this using simple applications of Chebyshev inequality. As mentioned earlier we pick a random function which puts values $-1$ or $1$ independently on every point of $\{-1,1\}^n$ with an equal probability of $1/2$. Recall that, $H_r$ and $I_r$ are random variables denoting the entropy and influence of a randomly chosen function. Let us define the event ${\cal E}_{FEI}$ indicating that the randomly chosen boolean function satisfies the FEI conjecture with the constant $C= 2 + 2\epsilon$ for $\epsilon > 0$. More precisely, ${\cal E}_{FEI}$ is the event that $H_r\leq (2 + 2\epsilon)I_r$. Our aim is to prove that
 ${\cal E}_{FEI}$ occurs with high probability. 
 From Fact \ref{a}, we have $H_r\leq n$.  Therefore,
\[
\mbox{Pr[$H_r >  n $]} = 0. 
\]
Consider the following events:
\[  
{\cal E}_1 := H_r > n
\]
\[
{\cal E}_2 := (2 + 2\epsilon) I_r  \leq n 
\]
Now if $H_r > (2 + 2\epsilon) I_r$, then either ${\cal E}_1$ or ${\cal E}_2$ happens. Therefore by union bound,
\[
 \mbox{Pr}[H_r > (2 + 2\epsilon) I_r] \leq \mbox{Pr}[{\cal E}_1] + \mbox{Pr}[{\cal E}_2]
\]

Since from Lemma \ref{lem1}, $\textbf{E}[I_r] = n/2$, $\textbf{E}[ (2 + 2\epsilon) I_r] = (1 + \epsilon) n$. Now we upper bound $\mbox{Pr}[{\cal E}_2]$.
\[ 
\mbox{Pr}[{\cal E}_2] = \mbox{Pr}[(2 + 2\epsilon) I_r \leq n] = \mbox{Pr}[(2 + 2\epsilon)I_r - (1+\epsilon) n \leq -\epsilon n]  \leq \mbox{Pr}[|(2 + 2\epsilon)I_r - (1+\epsilon) n| \geq \epsilon n] 
\]
Using Chebyshev Inequality,
\[
 \mbox{Pr}[|(2 + 2\epsilon)I_r - (1+\epsilon) n| \geq \epsilon n] \leq \frac{\textbf{Var}[(2 + 2\epsilon) I_r]}{\epsilon^2 n^2} =  4\left(1 + \frac{1}{\epsilon} \right)^2\frac{\textbf{Var}[I_r]}{n^2}.
\]
From Lemma \ref{lem2}, this in turn implies
\[
 \mbox{Pr}[H_r > (2 + 2\epsilon) I_r] \leq  4\left(1 + \frac{1}{\epsilon} \right)^2 \frac{1}{2^{n+1}n}
\]
Hence, it follows that $\mbox{Pr[${\cal E}_{FEI}$]} \geq 1 - 2\left(1 + \frac{1}{\epsilon} \right)^2 \frac{1}{2^{n}n}
$ with constant $C=2 + 2\epsilon$ for arbitrary constant $\epsilon > 0$.
\end{proof}

\section{Conclusion}
In this paper we gave a simple proof of the fact a random function
will almost surely satisfy the FEI conjecture for $C = 2 + \delta$ for $\delta > 0$. Although
our proof is non-constructive, this is the only doubly exponential
sized family for which it is known that the FEI conjecture is true.

 It would be interesting to get a large ($\omega(2^{poly(n)})$) explicit family of functions that satisfy FEI conjecture. One possible candidate is the class of functions $f(x_1, x_2, \dots, x_p)$ which are invariant under the action of the cyclic permutation group $C_p\leq \textrm{Sym}(p)$ where $p$ is prime. It can be verified that the size of this class is $2^{\frac{2^p-2}{p}+2}$. Because of the highly structured nature of the functions which are invariant under $C_p$ it might be plausible to verify the conjecture for these functions. 

\section{Acknowledgement}
We thank Kunal Dutta and Justin Salez for pointing out that our result can be extended to a high probability statement.


\begin{thebibliography}{99}

\bibitem{FG96} E. Friedgut and G. Kalai. {\it Every monotone graph property has a sharp threshold },
 {\sf Proceedings of AMS, 124(10)}, 1996, pp. 2993 - 3002.

 \bibitem{CT91} T. M. Cover and J. A. Thomas.  {\it Elements of Information Theory},  {\sf John Wiley \& Sons, Inc., N. Y.}, 1991.

\bibitem{GK} G. Kalai, {\it The entropy/influence conjecture },\\ 
{\sf http://terrytao.wordpress.com/2007/08/16/gil-kalai-the-entropyinfluence-conjecture/}

\bibitem{KMS11} N. Keller, E. Mossel, and T. Schlank,
 {\it A Note on the Entropy/Influence Conjecture},
 {\sf arXiv:1105.2651v1}, 2011
 
\bibitem{KLW10} A. Klivans, H. Lee, and A. Wan. {\it Mansour's Conjecture is true for random DNF formulas},
 {\sf COLT}, 2010, pp. 368-380.

 \bibitem{M94} Y. Mansour, {\it Learning Boolean functions via the Fourier transform}, {\sf Kluwer 
Academic Publishers}, 1994, pp. 391-424.

 \bibitem{DWZ11} R. O'donnell, J. Wright and Y. Zhou,
 {\it The Fourier Entropy-Influence Conjecture for certain classes of Boolean functions},
 {\sf ICALP}, 2011, pp. 330-341.
 
\end{thebibliography}
\end{document}